\documentclass[a4paper,11pt]{amsart}
\usepackage{amsmath,amsthm,amssymb}
\usepackage[mathscr]{eucal}
 \usepackage{cite}
\usepackage{upgreek}
\usepackage[bookmarks=false]{hyperref}
\usepackage{enumerate}
\usepackage{amsmath}
\usepackage{mathrsfs}
\usepackage{blindtext}
\usepackage{scrextend}
\usepackage{enumitem}
\usepackage{bm} 
\addtokomafont{labelinglabel}{\sffamily}
\usepackage{color}
\usepackage[at]{easylist}

\numberwithin{equation}{section}

\theoremstyle{plain}
\newtheorem{main}{Theorem}
\newtheorem{mcor}[main]{Corollary}

\newtheorem{theorem}{Theorem}[section]
\newtheorem{claim}[theorem]{Claim}
\newtheorem{lemma}[theorem]{Lemma}
\newtheorem{proposition}[theorem]{Proposition}
\newtheorem{corollary}[theorem]{Corollary}
\theoremstyle{definition}
\newtheorem{definition}[theorem]{Definition}
\newtheorem*{definition*}{Definition}
\newtheorem{example}[theorem]{Example}
\newtheorem{notation}[theorem]{Notation}

\newtheorem{remark}[theorem]{Remark}

\begin{document}

\title[Embedding universality for  II$_1$ factors with property (T)]
{Embedding universality for II$_1$ factors with property (T)}

\author[I. Chifan]{Ionut Chifan}
\address{14 MacLean Hall, Department of Mathematics, The University of Iowa, IA, 52242, U.S.A.}\email{ionut-chifan@uiowa.edu}

\author[D. Drimbe]{Daniel Drimbe}
\address{Department of Mathematics, KU Leuven, Celestijnenlaan 200b, B-3001 Leuven, Belgium}\email{daniel.drimbe@kuleuven.be}

\author[A. Ioana]{Adrian Ioana}
\address{Department of Mathematics, University of California San Diego, 9500 Gilman Drive, La Jolla, CA 92093, USA}\email{aioana@ucsd.edu}

{\thanks{I.C. was partially supported by NSF FRG Grant \#1854194; D.D. was supported by the postdoctoral fellowship fundamental research 12T5221N of the Research Foundation Flanders; A.I. was partially supported by NSF FRG Grant \#1854074.}}
\begin{abstract} 
We prove that every separable tracial von Neumann algebra embeds into a II$_1$ factor with property (T) which can be taken to have trivial outer automorphism and fundamental groups. We also establish an analogous result for the trivial extension over a non-atomic probability space of every countable p.m.p. equivalence relation. These results are obtained by using the class of wreath-like product groups introduced recently in \cite{CIOS21}.
 \end{abstract}

\maketitle

\section{Introduction and statement of main results}

Property (T), introduced by Kazhdan in \cite{Ka67}, is a fundamental representation-theoretic property of groups which has a wide spectrum of applications to several areas including  ergodic theory, group theory and operator algebras.
In operator algebras, it was first used by Connes in \cite{Co80} to prove that the II$_1$ factor $\text{L}(G)$ arising from a countable icc property (T) group $G$ has countable outer automorphism and fundamental groups.
 Connes and Jones then defined a notion of property (T) for II$_1$ factors and showed that a group II$_1$ factor $\text{L}(G)$ has property (T) if and only if  $G$ does \cite{CJ85}. In the last 20 years, property (T) has been instrumental in the advances made in the study of von Neumann algebras via Popa's deformation/rigidity theory (see the surveys \cite{Po06,Va10,Io18}).

Olshanskii  \cite{Ol95} and Delzant  \cite{De96} proved that every non-elementary hyperbolic group $H$ is {\it SQ-universal}: every countable group embeds into a quotient of $H$.
Since property (T) passes to quotients, taking $H$ to be a hyperbolic group with property (T) implies that every countable group embeds into a countable group with property (T),  as noted right after \cite[Th\'{e}or\`{e}me 3.5]{De96}. 

Our main result establishes an analogue of this fact for II$_1$ factors. Moreover, we have:

\begin{main}\label{A}
Any separable tracial von Neumann algebra $(M,\tau)$ embeds into a II$_1$ factor with property (T). Moreover,  the following hold:

\begin{enumerate}

\item For every acylindrically hyperbolic group $H$, $M$ embeds into a II$_1$ factor $Q$ which is generated by a representation $\pi:H\rightarrow\mathcal U(Q)$. Thus, if $H$ has property (T), then $Q$ has property (T).
\item $M$ embeds into a property (T) II$_1$ factor $P$ with $\emph{Out}(P)=\{e\}$ and $\mathcal F(P)=\{1\}$. 
\end{enumerate}
\end{main}
For a II$_1$ factor $P$, we denote by $\text{Out}(P)=\text{Aut}(P)/\text{Inn}(P)$ the {\it outer automorphism group of $P$} and by $\mathcal F(P)=\{\tau(e)/\tau(f)\mid\text{$e,f\in P$ projections, $ePe\cong fPf$}\}$
the {\it fundamental group of $P$} \cite{MvN43}.

Before putting Theorem \ref{A} into context, we note  that it leads to a characterization of separability for tracial von Neumann algebras. Indeed, since II$_1$ factors with property (T) are finitely generated by \cite{Po86} and thus separable, the following result is an immediate consequence of Theorem \ref{A}:

\begin{mcor}\label{B}
A tracial von Neumann algebra $(M,\tau)$ is separable if and only if it embeds into a II$_1$ factor with property (T).
\end{mcor}

Theorem \ref{A} implies that the family of separable II$_1$ factors with property (T) is {\it embedding universal}: any separable II$_1$ factor embeds into one with property (T).
This settles in the positive \cite[Question 6.3.21]{AGKE20}.

\begin{remark} (Non)embeddability results for II$_1$ factors (see \cite{PV21} for a survey of such results) lead to examples of (non)embedding universal families.
Embedding universal families include those consisting of all separable II$_1$ factors which have one of the following properties: are McDuff, are single generated, do not have property Gamma, are prime, do not have Cartan subalgebras. Indeed, if $M$ is a separable II$_1$ factor and $R$ is the hyperfinite II$_1$ factor, then $M\overline{\otimes}R$ is McDuff and thus single generated \cite{GP98}, while $M*\text{L}(\mathbb Z)$ is a non-Gamma II$_1$ factor which is prime \cite{Pe06} and has no Cartan subalgebras \cite{Io12}. 
On the other hand, the family of  II$_1$ factors with Haagerup's property \cite{CJ85} and any countable family of separable II$_1$ factors \cite{Oz02} are not embedding universal.
It is open whether the family of all II$_1$ factors $\text{L}(G)$ arising from icc countable groups $G$ is embedding universal.
\end{remark}

\begin{remark}
The Connes Embedding Problem (CEP) asks if every separable II$_1$ factor embeds into the ultrapower, $R^\omega$, of $R$, where $\omega$ is a free ultrafilter on $\mathbb N$ \cite{Co76}. By Theorem \ref{A}, the CEP is equivalent to asking whether every  property (T) II$_1$ factor embeds into $R^\omega$. Thus, a negative answer to the CEP, announced recently in \cite{JNVWY20}, implies the existence of a non-embeddable II$_1$ factor with property (T).
\end{remark}

 Part (1) of Theorem \ref{A} can be viewed as a II$_1$ factor analogue of the SQ-universality of acylindrically hyperbolic groups established in \cite[Theorem 8.1(a)]{DGO11}. We refer to \cite[Section 3.2]{CIOS21} for the definition of acylindrically hyperbolic groups and note  that this class contains all non-elementary hyperbolic groups. In particular, Theorem \ref{A}(1) applies to cocompact lattices $H$ in the rank one simple Lie group $\text{Sp}(n,1)$, $n\geq 2$, as such $H$ are hyperbolic and have property (T). Theorem \ref{A}(1) highlights a striking difference between the type II$_1$ representations of $H$ and those of icc lattices $G$ in higher rank simple Lie groups (e.g. $\text{SL}_m(\mathbb R)$, $m\geq 3$).  Work of Peterson \cite{Pe14} (see also \cite{Be06,BH19}) shows that $\text{L}(G)$ is the only II$_1$ factor generated by a representation of $G$. In contrast, Theorem \ref{A}(1) shows that the family of  II$_1$ factors generated by representations of $H$ is embedding universal.

An immediate consequence of part (2) of Theorem \ref{A}, which is new to our knowledge, is that every separable II$_1$ factor embeds into a separable II$_1$ factor with trivial outer automorphism group.
To further discuss Theorem \ref{A}(2), denote by $\bm{\mathcal {T}}$ the family of all II$_1$ factors $P$ with property (T) which satisfy $\text{Out}(P)=\{e\}$ and $\mathcal F(P)=\{1\}$. In a major breakthrough in \cite{Po01}, Popa discovered the first examples of  II$_1$ factors $P$ with $\mathcal F(P)=\{1\}$. The existence of II$_1$ factors $P$ with $\text{Out}(P)=\{e\}$ and $\mathcal F(P)=\{1\}$ was obtained in \cite{IPP05}. However, none of these II$_1$ factors have property (T). Only recently, property (T) II$_1$ factors $P$ with  $\mathcal F(P)=\{1\}$ were found in \cite{CDHK20}.
Most recently, the fact that $\bm{\mathcal{T}}$ is nonempty was proved in \cite{CIOS21}. To elaborate on the last result, we note that Popa's strengthening of Connes' rigidity conjecture (see \cite[Section 3]{Po06}) predicts that $\text{L}(G)\in \bm{\mathcal{T}}$, whenever $G$ is an icc property (T) group with $\text{Out}(G)=\{e\}$ and no characters. Confirming this conjecture for an uncountable class of groups, it was shown in \cite[Corollary 2.7]{CIOS21} that $\bm{\mathcal{T}}$ contains a continuum $\{\text{L}(G_i)\}_{i\in I}$ of nonisomorphic property (T) group II$_1$ factors. In particular, $\bm{\mathcal {T}}$ is uncountable.
Theorem \ref{A}(2) moreover shows that $\bm{\mathcal{T}}$ is embedding universal. Our next main result further strengthens this fact. Denote by $\text{\bf II}_1$ the family of all separable II$_1$ factors.

\begin{main}\label{C}
There is a map $\emph{\bf II}_1\times\mathbb R\ni (M,s)\mapsto P_{(M,s)}\in \bm{\mathcal{T}}$ such that
 $M$ embeds into $P_{(M,s)}$, and if $P_{(M,s)}$ is stably isomorphic to $P_{(N,t)}$ then $M\cong N$ and $s=t$, for every $(M,s),(N,t)\in\emph{\bf II}_1\times\mathbb R$.
\end{main}

Next, we state an analogue of  Theorem \ref{A}(2) for countable p.m.p. equivalence relations using Zimmer's notion of property (T) \cite{Zi81} in this context. 
While it remains open whether any countable p.m.p. equivalence relation $\mathcal R$ is contained into an ergodic countable p.m.p. equivalence relation with property (T), we prove this for the trivial extension of $\mathcal R$ over a non-atomic probability space:

\begin{main}\label{D}
Let $\mathcal R$ be a countable p.m.p. equivalence relation on a probability space $(X,\mu)$. Let $\Delta_Y=\{(y,y)\mid y\in Y\}$ be the trivial equivalence relation on a non-atomic probability space $(Y,\nu)$.
Then $\mathcal R\times\Delta_Y$ embeds into a countable ergodic p.m.p. equivalence relation $\mathcal S$ on $X\times Y$ such that $\emph{L}(\mathcal S)\in\bm{\mathcal{T}}$.
In particular, $\mathcal S$ has property (T), $\emph{Out}(\mathcal S)=\{e\}$ and $\mathcal F(\mathcal S)=\{1\}$.
\end{main}

Finally, we mention two applications of Theorem \ref{A} to the structure of {\it infinitely generic} II$_1$ factors. First, extending the main result of \cite{Go20}, we prove that, given an infinitely generic II$_1$ factor $Q$, any separable full II$_1$ factor $P$ admits an embedding $P\hookrightarrow Q^\omega$ with factorial relative commutant (see Theorem \ref{Go20}).
Secondly, we prove that any infinitely generic II$_1$ factor $Q$ is super McDuff, i.e., its central sequence algebra $Q'\cap Q^\omega$ is a II$_1$ factor (see Theorem \ref{igsMc}).

\subsection*{Comments on the proof of Theorem \ref{A}} 
The proof of Theorem \ref{A} relies heavily on the class of wreath-like product groups introduced recently by Osin, Sun and two of the authors in \cite{CIOS21}. Specifically, to prove the main assertion of Theorem \ref{A}, we combine the following two facts:
\begin{enumerate}[label=(\alph*)]
\item Let $G\in\mathcal W\mathcal R(A,B)$ be a wreath-like product of two infinite groups $A$ and $B$. Then any homomorphism $\pi:A\rightarrow\mathcal U(M)$  generating a II$_1$ factor $M$ extends to a homomorphism $\widetilde\pi:G\rightarrow\mathcal U(N)$ which generates a  II$_1$ factor $N$ containing $M$ (see Proposition \ref{extend}). This new II$_1$ factor $N$ is a wreath-like product of $M$ and $B$, in a sense defined in Section \ref{wrelike}.

 \item There exists a property (T) group $G\in\mathcal W\mathcal R(\mathbb F_3,B)$, for some group $B$.

\end{enumerate}
Fact (a) is a consequence of the existence of a  special cocycle semidirect product decomposition of $G$ observed in a preliminary version of \cite{CIOS21} (see Lemma \ref{special}). Fact (b) was proved in \cite{CIOS21} (see Theorem \ref{wreath}).

If $M$ is a II$_1$ factor generated by $3$ unitaries, then there is a homomorphism $\pi:\mathbb F_3\rightarrow\mathcal U(M)$ generating $M$. 
Combining (a) and (b) implies the existence of a II$_1$ factor $N$ which contains $M$ and is generated by a representation of a property (T) group  and thus has property (T). Since every separable tracial von Neumann algebra is contained in a II$_1$ factor generated by $3$ unitaries (in fact, in a II$_1$ factor which is single generated and thus generated by $2$ unitaries \cite{GP98}), we conclude that every separable tracial von Neumann algebra is contained in a II$_1$ factor with property (T). This implies the main assertion of Theorem \ref{A}.
Part (1) of Theorem \ref{A} also follows since, as shown in \cite{CIOS21} (see Theorem \ref{wreath}),  one can take $G$ in (b) to be a quotient of any  acylindrically hyperbolic group.
Finally, to prove part (2) of Theorem \ref{A} we use that in (b) one can take $B$ to be an icc hyperbolic group with property (T) and trivial outer automorphism group.
 Additionally, we adapt techniques from Popa's deformation/rigidity theory used in \cite{CIOS21} to show that if $G\in\mathcal W\mathcal R(A,B)$ has property (T) and no characters, then any II$_1$ factor $N$ constructed as in (a) has trivial  outer automorphism and fundamental groups (see Theorem \ref{out}). 
\section{Preliminaries}

\subsection{Wreath-like product groups}\label{wrlike}
The proofs of our main results rely on the class of wreath-like product groups introduced in \cite{CIOS21}:

\begin{definition}\cite{CIOS21}
A group $G$ is called a {\it wreath-like product} of two groups $A$ and $B$ if it is an extension of the form $\{e\}\rightarrow\oplus_{b\in B}A_b\rightarrow G\xrightarrow{\varepsilon}B\rightarrow\{e\}$ such that $A_b\cong A$ and $gA_bg^{-1}=A_{\varepsilon(g)b}$, for every $b\in B$ and $g\in G$. The set of all wreath-like products of $A$ and $B$ is denoted by $\mathcal W\mathcal R(A,B)$.
\end{definition}

It is clear that the usual wreath product group $A\wr B=A^{(B)}\rtimes B$ belongs to $\mathcal W\mathcal R(A,B)$. A remarkable fact established in \cite{CIOS21} is that although wreath products never have property (T), there is an abundance of wreath-like products with property (T). This is illustrated by the following result.

\begin{theorem}\emph{\cite{CIOS21}}\label{wreath}
Let $A$ be any finitely generated group.
\begin{enumerate}
\item There exists  a property (T) group $G$ with no nontrivial characters which belongs to $\mathcal W\mathcal R(A,B)$, for some icc property (T) hyperbolic group $B$ with $\emph{Out}(B)=\{e\}$.
\item Every acylindrically hyperbolic group $H$ admits a quotient which belongs to $\mathcal W\mathcal R(A,C)$, for some group $C$.
\end{enumerate}
\end{theorem}
For parts (1) and (2), see \cite[Theorem 6.9]{CIOS21} and \cite[Theorem 4.20]{CIOS21}, respectively.

\subsection{Tracial von Neumann algebras} A {\it tracial von Neumann algebra} is a pair $(M,\tau)$ consisting of a von Neumann algebra $M$ and a faithful normal tracial state $\tau:M\rightarrow\mathbb C$. 
We denote by $\mathcal U(M)$ the group of unitary elements of $M$ and by $\text{Aut}(M)$ the group of $\tau$-preserving automorphisms of $M$. 
By von Neumann's bicommutant theorem, if $S\subset \mathcal U(M)$, then $S''\subset M$ is equal to the von Neumann algebra generated by $S$. 
For $u\in\mathcal U(M)$, we denote by $\text{Ad}(u)\in\text{Aut}(M)$ the inner automorphism of $M$ given by $\text{Ad}(u)(x)=uxu^*$. The group of all inner automorphisms of $M$ is denoted by $\text{Inn}(M)$. For  a set $I$, we denote by $(M^I,\tau)$ the tensor product tracial von Neumann algebra $\overline{\otimes}_{i\in I}(M, \tau)$. For $J\subset I$, we view $M^J\subset M^I$, in the natural way. For $i\in I$, we write $M^i$ instead of $M^{\{i\}}$. 

A tracial von Neumann algebra $(M,\tau)$ is a {\it II$_1$ factor} if it is infinite dimensional and has trivial center. A II$_1$ factor $M$ has {\it property (T)}, in the sense of Connes and Jones \cite{CJ85}, if there exist $F\subset M$ finite and $\delta>0$ such that if $\mathcal H$ is a Hilbert $M$-bimodule and $\xi\in\mathcal H$ is a unit vector with $\max_{x\in F}\|x\xi-\xi x\|<\delta$, then there exists $\eta\in\mathcal H$, $\eta\not=0$, such that $y\eta=\eta y$, for every $y\in M$.
If $M$ is a II$_1$ factor such that $M=\pi(G)''$, for a property (T) group $G$ and a homomorphism $\pi:G\rightarrow\mathcal U(M)$, then $M$ has property (T) (see the proof of \cite[Theorem 2]{CJ85} or \cite[Theorem 4.1.7]{Po86}).



\subsection{Equivalence relations}\label{ER} Let $(X,\mu)$ be a probability space, which we will always assume to be standard. Let $\mathcal R$ be a countable p.m.p. equivalence relation on $(X,\mu)$ \cite{FM77}. We endow $\mathcal R$ with the $\sigma$-finite Borel measure $\nu_{\mathcal R}$ defined by $\nu_{\mathcal R}(A)=\int_{X} |A^x|\;\text{d}\mu(x)$, where $A^x=\{y\in X|(x,y)\in A \}$.

For $x\in X$, we denote by $[x]_{\mathcal R}$ its $\mathcal R$-equivalence class. The {\it automorphism group} of $\mathcal R$, denoted $\text{Aut}(\mathcal R)$, is the group of  all $\theta\in\text{Aut}(X,\mu)$ such that $\theta([x]_{\mathcal R})=[\theta(x)]_{\mathcal R}$, for $\mu$-almost every $x\in X$. The {\it full group} of $\mathcal R$, denoted $[\mathcal R]$, consists of all $\theta\in\text{Aut}(\mathcal R)$ such that $\theta(x)\in [x]_{\mathcal R}$, for $\mu$-almost every $x\in X$. Note that $[\mathcal R]$ is a normal subgroup of $\text{Aut}(\mathcal R)$.
The {\it outer automorphism} of $\mathcal R$ is  defined as $\text{Out}(\mathcal R)=\text{Aut}(\mathcal R)/[\mathcal R]$.
The {\it fundamental group} of $\mathcal R$ is the set of all quotients $\mu(A)/\mu(B)\in (0,+\infty)$, where $A,B\subset X$ are non-null measurable sets such that $\mathcal R\cap (A\times A)\cong\mathcal R\cap (B\times B)$.

Associated to $\mathcal R$ is a tracial von Neumann algebra $\text{L}(\mathcal R)$ \cite{FM77}. This  is generated by a copy of $\text{L}^\infty(X)$ and the image of a homomorphism $u:[\mathcal R]\rightarrow\mathcal U(\text{L}(\mathcal R))$ such that $u_\theta fu_\theta^*=f\circ\theta^{-1}$, for every $f\in\text{L}^\infty(X)$ and $\theta\in[\mathcal R]$. We have a homomorphism $$i:\text{Aut}(\mathcal R)\rightarrow\text{Aut}(\text{L}(\mathcal R))$$ given by $i(\varphi)(u_\theta)=u_{\varphi\theta\varphi^{-1}}$ and $i(\varphi)(f)=f\circ\varphi^{-1}$, for every $\varphi\in\text{Aut}(\mathcal R)$, $\theta\in [\mathcal R]$ and $f\in \text{L}^\infty(X)$.
Then $i([\mathcal R])\subset\text{Inn}(\text{L}(\mathcal R))$ and $i$ descends to an injective homomorphism $\text{Out}(\mathcal R)\hookrightarrow\text{Out}(\text{L}(\mathcal R))$. Moreover, we have that $\mathcal F(\mathcal R)\subset\mathcal F(\text{L}(\mathcal R))$.

In \cite{Zi81}, Zimmer introduced a notion of property (T) for countable p.m.p. equivalence relations. 
Let $\mathcal R$ be a countable ergodic p.m.p. equivalence relation on a probability space $(X,\mu)$. If $\text{L}(\mathcal R)$ has property (T), then $\mathcal R$ has property (T).
This is because property (T) for $\mathcal R$ is equivalent to $\text{L}(\mathcal R)$ having property (T) relative to $\text{L}^\infty(X)$ in the sense of \cite[Definition 4.1.3]{Po86}, and latter condition holds if  $\text{L}(\mathcal R)$ has property (T).

\subsection{Cocycle actions} In this subsection, we discuss several notions of cocycle actions on von Neumann algebras, groups and equivalence relations.

\begin{definition}
A {\it cocycle action} of a group $B$ on a tracial von Neumann $(M,\tau)$ is a pair $(\beta, w)$ consisting of two maps $\beta:B\rightarrow\text{Aut}(M)$ and $w:B\times B\rightarrow\mathcal U(M)$ which satisfy the following 
\begin{enumerate}
\item $\beta_b\beta_c=\text{Ad}(w_{b,c})\beta_{bc}$, for every $b,c\in B$, 
\item $w_{b,c}w_{bc,d}=\beta_b(w_{c,d})w_{b,cd}$, for every $b,c,d\in B$, and
\item $w_{b,e}=w_{e,b}=1$, for every $b\in B$.
\end{enumerate}
\end{definition}

\begin{definition}
Let $(\beta,w)$ be a cocycle action of a group $B$ on a tracial von Neumann algebra $(M,\tau)$.
The {\it cocycle crossed product} von Neumann algebra $M\rtimes_{\beta,w}B$ is a tracial von Neumann algebra which is  generated by a copy of $M$ and unitary elements $\{u_b\}_{b\in B}$ such that
$u_bxu_b^*=\beta_b(x)$, $u_bu_c=w_{b,c}u_{bc}$ and $\tau(xu_b)=\tau(x)\delta_{b,e}$, for every $b,c\in B$ and $x\in M$.
\end{definition}

The above notions have well-known analogues for groups (see, e.g., \cite[pages 104-105]{Br82}):
\begin{definition}
A {\it cocycle action} of a group $B$ on a group $A$ is a pair $(\alpha, v)$ consisting of two maps $\alpha:B\rightarrow\text{Aut}(A)$ and $v:B\times B\rightarrow A$ which satisfy the following 
\begin{enumerate}
\item $\alpha_b\alpha_c=\text{Ad}(v_{b,c})\alpha_{bc}$, for every $b,c\in B$, 
\item $v_{b,c}v_{bc,d}=\alpha_b(v_{c,d})v_{b,cd}$, for every $b,c,d\in B$, and
\item $v_{b,e}=v_{e,b}=e$, for every $b\in B$.
\end{enumerate}
\end{definition}

\begin{definition}\label{csproduct}
 Let $(\alpha, v)$ be a cocycle action of a group $B$ on group $A$.   Then the set $A\times B$ endowed with the unit $e=(e,e)$ and the multiplication operation $(x,b)\cdot (y,c)=(x\alpha_b(y)v_{b,c},bc)$ is a group, denoted $A\rtimes_{\alpha,v}B$, and called the {\it cocycle semidirect product} group. 
Moreover, we have a short exact sequence $\{e\}\rightarrow A\xrightarrow{i} A\rtimes_{\alpha,v}B\xrightarrow{\gamma} B\rightarrow\{e\}$, where $i(a)=(a,e)$ and $\gamma(a,b)=b$.\end{definition}

While the above terminology (cocycle actions and cocycle semidirect products) is not standard in group theory, we adopt it in this paper by analogy with the von Neumann algebra case.

Conversely, as is well-known and easy to see, any extension arises as a cocycle semidirect product:

\begin{lemma}\label{semi}
Let $A,B$ be groups and consider a short exact sequence $\{e\}\rightarrow A\xrightarrow{i} G\xrightarrow{\gamma} B\rightarrow \{e\}$. 
Let $k:B\rightarrow G$ be a map such that $\gamma(k_b)=b$, for every $b\in B$. Define $\alpha:B\rightarrow\emph{Aut}(A)$ and $v:B\times B\rightarrow A$ by letting $\alpha_b=\emph{Ad}(k_b)$ and $v_{b,c}=k_bk_ck_{bc}^{-1}$, for every $b,c\in B$.
Then $(\alpha,v)$ is a cocycle action of $B$ on $A$ and $\delta:A\rtimes_{\alpha,v}B\rightarrow G$ given by $\delta(a,b)=i(a)k_b$ is an isomorphism.
\end{lemma}

The next lemma, whose proof is straightforward, relates the above two notions of cocycle actions.

\begin{lemma}\label{extension}
Let $(\alpha, v)$ be a cocycle action of a group $B$ on a group $A$. Let  $(\beta,w)$ be a cocycle action of $B$ on a tracial von Neumann algebra $(M,\tau)$ and $\pi:A\rightarrow\mathcal U(M)$ be a homomorphism. Assume that $\pi\circ\alpha_b=\beta_b\circ\pi$ and $w_{b,c}=\pi(v_{b,c})$, for every $b,c\in B$.
Then $\pi$ extends to a homomorphism $\widetilde\pi:A\rtimes_{\alpha,v}B\rightarrow \mathcal U(M\rtimes_{\beta,w}B)$ given by $\widetilde\pi(a,b)=\pi(a)u_b$, for every $a\in A,b\in B$.
\end{lemma}

\begin{definition} Let $\mathcal R$ be a countable p.m.p. equivalence relation on a probability space $(X,\mu)$.
A {\it cocycle action} of a countable group $B$ on $\mathcal R$ is a pair $(\gamma, \omega)$ consisting of two maps $\gamma:B\rightarrow\text{Aut}(\mathcal R)$ and $\omega:B\times B\rightarrow [\mathcal R]$ which satisfy the following 
\begin{enumerate}
\item $\gamma_b\gamma_c=\omega_{b,c}\gamma_{bc}$, for every $b,c\in B$, 
\item $\omega_{b,c}\omega_{bc,d}=\text{Ad}(\gamma_b)(\omega_{c,d})\omega_{b,cd}$, for every $b,c,d\in B$, and
\item $\omega_{b,e}=\omega_{e,b}=\text{Id}_X$, for every $b\in B$.
\end{enumerate}
\end{definition}
The following lemma, which uses the notation from Subsection \ref{ER}, is immediate:
\begin{lemma}\label{cocycleER}  Let $\mathcal R$ be a countable p.m.p. equivalence relation on a probability space $(X,\mu)$ and $(\gamma,\omega)$ be a cocycle action  of a countable group $B$ on $\mathcal R$ such that $\mu(\{x\in X\mid \gamma_b(x)\in [x]_{\mathcal R}\})=0$, for every $b\in B\setminus\{e\}$.
 Let $\mathcal S$ be the smallest equivalence relation on $X$ which contains $\mathcal R$ and the graph of $\gamma_b$, for every $b\in B$. Then $(i\circ\gamma,u\circ\omega)$ is a cocycle action of $B$ on $\emph{L}(\mathcal R)$ and $\emph{L}(\mathcal S)\cong\emph{L}(\mathcal R)\rtimes_{i\circ\gamma,u\circ\omega}B$.
\end{lemma}

{\it Proof.} The fact that $(i\circ\gamma,u\circ\omega)$ is a cocycle action of $B$ on $\emph{L}(\mathcal R)$ follows trivially. To prove the second assertion, we notice that the hypothesis implies that $[x]_{\mathcal S}=\sqcup_{g\in B} [\gamma_{g}(x)]_{\mathcal R}$, for almost every $x\in X$.
We claim that the map $\phi: (\mathcal R\times B, \nu_{\mathcal R}\times c )\to \mathcal (S,\nu_{\mathcal S})$ given by $\phi(x,y,g)=(x,\gamma_{g^{-1}}(y))$, for all $(x,y)\in \mathcal R$ and $g\in B$,
is a measure space isomorphism, where  $c$ is the counting measure on $B$. 

To justify this claim, it suffices to show that $ (\nu_{\mathcal R}\times c)(\phi^{-1}(A))=\nu_{\mathcal S}(A)$, for any Borel set $A\subset\mathcal S$. 
To this end, let $C_g=\{(x,y)\in\mathcal R| (x,\gamma_{g^{-1}}(y))\in A \}$, for $g\in B$. Then we have that
\begin{equation}\label{o1}
\begin{array}{rcl}
    (\nu_{\mathcal R}\times c)(\phi^{-1}(A))  &=&   \sum_{g\in B} \int_{\mathcal R} 1_A(x,\gamma_{g^{-1}}(y))\;\text{d}\nu_{\mathcal R}(x,y)=\sum_{g\in B} \nu_{\mathcal R}(C_g) \\
    &=& \sum_{g\in B} \int_{X} |C_g\cap [x]_{\mathcal R}|\;\text{d}\mu(x).
\end{array}
\end{equation}
Note that for almost every $x\in X$ and $y\in A\cap [x]_{\mathcal S}$, there is a unique $g\in B$ such that $(x,y)\in A$ and $(x,\gamma_{g}(y))\in \mathcal R.$  This shows that $A\cap [x]_{\mathcal S}=\sqcup_{g\in B} \gamma_{g^{-1}}(C_g\cap [x]_{\mathcal R})$, and therefore,
\begin{equation}\label{o2}
\nu_{\mathcal S}(A)=\int_X |A\cap [x]_{\mathcal S}|\;\text{d}\mu(x)=\int_X \sum_{g\in B} |C_g\cap [x]_{\mathcal R}|\;\text{d}\mu(x).  
\end{equation}
Equations \eqref{o1} and \eqref{o2} together with Fubini's theorem show that $ (\nu_{\mathcal R}\times c)(\phi^{-1}(A))=\nu_{\mathcal S}(A)$, proving that $\phi$ is a measure space isomorphism.
Let $V:\text{L}^2(\mathcal S,\nu_{\mathcal S})\to \text{L}^2(\mathcal R\times B, \nu_{\mathcal R}\times c)$ be the unitary operator defined by $V\xi=\xi\circ \phi$, for any $\xi\in \text{L}^2(\mathcal S,\nu_{\mathcal S})$. 

Denote by $u:[\mathcal R]\rightarrow\mathcal U(\text{L}(\mathcal R))$ and $w:[\mathcal S]\rightarrow\mathcal U(\text{L}(\mathcal S))$ the homomorphisms that together with a copy of $\text{L}^\infty(X)$ generate the von Neumann algebras $\text{L}(\mathcal R)$ and $\text{L}(\mathcal S)$, respectively. Denote by $\{u_g\}_{g\in B}$ the canonical unitaries that implement the cocycle action  $(i\circ\gamma,u\circ\omega)$ of $B$ on $\text{L}(\mathcal R)$. Since $\mathcal R$ is a subequivalence relation of $\mathcal S$, we have $[\mathcal R]\subset [\mathcal S]$. It is easy to see that ${\rm Ad}(V^*)$ preserves $\text{L}^\infty(X)$, maps $u_{\theta}$ to $w_{\theta}$ for $\theta\in [\mathcal R]$ and $u_g$ to $w_g$ for $g\in B$. This implies that $V^*( \text{L}(\mathcal R)\rtimes_{i\circ\gamma,u\circ\omega}B) V \subset \text{L}(\mathcal S)$. Since $\mathcal S$ is generated by $[\mathcal R]$ and $B$,  $\text{L}(\mathcal S)$ is generated by $\text{L}^{\infty}(X)$ together with $\{w_\theta\mid \theta\in [\mathcal R]\cup B\}$. Thus, we deduce that $V^*( \text{L}(\mathcal R)\rtimes_{i\circ\gamma,u\circ\omega}B) V = \text{L}(\mathcal S)$, which proves the lemma.
\hfill$\blacksquare$
 

\subsection{Wreath-like product groups as cocycle semidirect products} 
In this subsection we show that wreath-like products admit a special cocycle semidirect product decomposition.  
For groups $A$ and $B$, we denote by $\sigma$ the {\it shift} action of $B$ on $A^B=\prod_{b\in B}A$ given by $\sigma_b(x)=(x_{b^{-1}c})_{c\in B}$, for every $x=(x_c)_{c\in B}\in A^B$ and $b\in B$. Note that $\sigma$ leaves invariant the normal subgroup $A^{(B)}=\oplus_{b\in B}A$ of $A^B$.
We start by recording the following direct consequence of Lemma \ref{semi}:

\begin{corollary}\label{wre}
Let $A,B$ be groups. A group $G$ belongs to $\mathcal W\mathcal R(A,B)$ if and only if it is isomorphic to $A^{(B)}\rtimes_{\alpha,v}B$, for a cocycle action $(\alpha,v)$ on $B$ on $A^{(B)}$ such that $\alpha_b(A_c)=A_{bc}$, for every $b,c\in B$.
\end{corollary}

We continue with an example of cocycle actions satisfying Corollary \ref{wre}.

\begin{lemma}\label{ex}
Let $A,B$ be groups. Let $\rho:B\rightarrow A^B$ be a map such that $v_{b,c}:=\rho_b\sigma_b(\rho_c)\rho_{bc}^{-1}\in A^{(B)}$, for every $b,c\in B$, and $\rho_e=e$.
For $b\in B$, define $\alpha_b:=\emph{Ad}(\rho_b)\sigma_b\in\emph{Aut}(A^{(B)})$. Then $(\alpha,v)$ is a cocycle action of $B$ on $A^{(B)}$ and the cocycle semidirect product $G=A^{(B)}\rtimes_{\alpha,v}B$ belongs $\mathcal W\mathcal R(A,B)$.
\end{lemma}

{\it Proof.}
It is straightforward to check that $(\alpha, v)$ is a cocycle action. 
Since  clearly $\alpha_b(A_c)=A_{bc}$, for every $b,c\in B$, Corollary \ref{wre}  implies that $G\in\mathcal W\mathcal R(A,B)$.
\hfill$\blacksquare$

The following key result, obtained in a preliminary version of \cite{CIOS21}, shows that any wreath-like product group admits a decomposition as in Lemma \ref{ex}:

\begin{lemma}\emph{\cite{CIOS21}}\label{special} Let $G\in\mathcal W\mathcal R(A,B)$, for groups $A,B$. Then there is $\rho:B\rightarrow A^B$ such that $v_{b,c}:=\rho_b\sigma_b(\rho_c)\rho_{bc}^{-1}\in A^{(B)}$, for every $b,c\in B$, and $\rho_e=e$, and letting  $\alpha_b:=\emph{Ad}(\rho_b)\sigma_b\in\emph{Aut}(A^{(B)})$, for every $b\in B$, we have that  $G\cong A^{(B)}\rtimes_{\alpha,v}B$.
\end{lemma}

{\it Proof.}
We have a short exact sequence $\{e\}\rightarrow\oplus_{b\in B}A_b\rightarrow G\xrightarrow{\varepsilon}B\rightarrow\{e\}$ such that $A_b\cong A$ and $gA_bg^{-1}=A_{\varepsilon(g)b}$, for every $b\in B$ and $g\in G$.  For $b\in B$, fix $k_b\in G$ with $\varepsilon(k_b)=b$. Take $k_e=e$.

Let $\varphi_b:A_e\rightarrow A_b$ be the isomorphism given by $\varphi_b(x)=k_bxk_b^{-1}$. Define $\varphi:A_e^{(B)}\rightarrow \oplus_{b\in B}A_b$ by $$\varphi=\oplus_{b\in B}\varphi_b.$$
If $g\in G$ and $b\in B$, then $\varepsilon(k_b^{-1}gk_{\varepsilon(g)^{-1}b})=e$ and thus $k_b^{-1}gk_{\varepsilon(g)^{-1}b}\in \oplus_{b\in B}A_b$. Denote by $\pi:\oplus_{b\in B}A_b\rightarrow A_e$ the quotient homomorphism given by $\pi((x_b)_{b\in B})=x_e$. Define a map $\tau:G\rightarrow A_e^B$ by letting 
$$\text{$\tau_g=(\pi(k_b^{-1}gk_{\varepsilon(g)^{-1}b}))_{b\in B}$, for every $g\in G$.}$$
If $g\in G$ and $x\in \oplus_{b\in B}A_b$, then $\tau_{gx}\tau_g^{-1}=(\pi(k_b^{-1}gxg^{-1}k_b))_{b\in B}$.
Write $gxg^{-1}=(y_c)_{c\in B}\in \oplus_{b\in B}A_b$. Since $k_b^{-1}y_ck_b\in A_{b^{-1}c}$, we get that $\pi(k_b^{-1}gxg^{-1}k_b)=k_b^{-1}y_bk_b$. Thus, $\pi(k_b^{-1}gxg^{-1}k_b)=e$ if $y_b=e$ and $\varphi_b(\pi(k_b^{-1}gxg^{-1}k_b))=y_b$, for every $b\in B$.   This implies that 
\begin{equation}\label{tau}
\text{$\tau_{gx}\tau_g^{-1}\in A_e^{(B)}$ and $\varphi(\tau_{gx}\tau_g^{-1})=gxg^{-1}$, for every $g\in G$ and $x\in \oplus_{b\in B}A_b$.}
\end{equation}
Now, a direct calculation shows that $\tau$ is a $1$-cocycle for  the action $\sigma\circ\varepsilon:G\rightarrow\text{Aut}(A_e^B)$:
 \begin{equation}\label{cocycletau} \text{$\tau_{g_1}\sigma_{\varepsilon(g_1)}(\tau_{g_2})=\tau_{g_1g_2}$, for every $g_1,g_2\in G$.}
\end{equation}
This follows by using that $(\sigma_{\varepsilon(g_1)}(\tau_{g_2}))_b=(\tau_{g_2})_{\varepsilon(g_1)^{-1}b}=\pi(k_{\varepsilon(g_1)^{-1}b}^{-1}g_2k_{\varepsilon(g_1g_2)^{-1}b})$, for every $b\in B$.

Moreover, as automorphisms of $A_e^B$, we have that 
\begin{equation}\label{conjugat} \text{$\varphi^{-1}\circ\text{Ad}(g)\circ\varphi=\text{Ad}(\tau_g)\circ\sigma_{\varepsilon(g)}$, for every $g\in G$.}
\end{equation}
To justify this, note that $\pi(y)a\pi(y)^{-1}=yay^{-1}$, for every $a\in A_e$ and $y=(y_b)_{b\in B}\in\oplus_{b\in B}A_b$.
Thus, if $x=(x_b)_{b\in B}\in A_e^{(B)}$, then $$((\text{Ad}(\tau_g)\circ\sigma_{\varepsilon(g)})(x))_b=\text{Ad}(\pi(k_b^{-1}gk_{\varepsilon(g)^{-1}b}))(x_{\varepsilon(g)^{-1}b})=\text{Ad}(k_b^{-1}gk_{\varepsilon(g)^{-1}b})(x_{\varepsilon(g)^{-1}b}),$$ and so $((\varphi\circ\text{Ad}(\tau_g)\circ\sigma_{\varepsilon(g)})(x))_b=\text{Ad}(gk_{\varepsilon(g)^{-1}b})(x_{\varepsilon(g)^{-1}b})=((\text{Ad}(g)\circ\varphi)(x))_b$. This proves \eqref{conjugat}.

Define $\rho:B\rightarrow A_e^B$ by  $\rho_b=\tau_{k_b}$, for $b\in B$. Let $v:B\times B\rightarrow A_e^B$ and $\alpha:B\rightarrow\text{Aut}(A_e^{(B)})$ be given by $v_{b,c}=\rho_b\sigma_b(\rho_c)\rho_{bc}^{-1}$ and $\alpha_b=\text{Ad}(\rho_b)\sigma_b$, for $b,c\in B$.
Then using \eqref{tau} and \eqref{cocycletau} we get \begin{equation}\label{v}
\text{$v_{b,c}=\tau_{k_b}\sigma_{\varepsilon(k_b)}(\tau_{k_c})\tau_{k_{bc}}^{-1}=\tau_{k_bk_c}\tau_{k_{bc}}^{-1}=\varphi^{-1}(k_bk_bk_{bc}^{-1})\in A_e^{(B)}$, for every $b,c\in B$.}
\end{equation}

Define $\theta:A_e^{(B)}\rtimes_{\alpha, v}B\rightarrow G$ by letting $\theta(x,b)=\varphi(x)k_b$.
If $(x,b),(y,c)\in A_e^{(B)}\rtimes_{\alpha,v}B$, then \eqref{conjugat} implies that $\text{Ad}(k_b)\circ\varphi=\varphi\circ\text{Ad}(\tau_{k_b})\circ\sigma_{\varepsilon(k_b)}=\varphi\circ\alpha_b$.  In combination with \eqref{v}, we derive that
\begin{align*}\theta(x,b)\theta(y,c)=\varphi(x)k_b\varphi(y)k_c&=\varphi(x)\text{Ad}(k_b)(\varphi(y))k_bk_c\\&=\varphi(x\alpha_b(y))k_bk_c=\varphi(x\alpha_b(y)v_{b,c})k_{bc}=\theta((x,b)\cdot (y,c)). \end{align*}
This proves that $\theta$ is a homomorphism. Since $\theta$ is bijective, it follows that $\theta$ is an isomorphism.
\hfill$\blacksquare$

\begin{remark}
As a consequence of \cite[Lemma 4.1]{CIOS21}, if $\psi:A\rightarrow A_0$ is an onto homomorphism and $G\in\mathcal W\mathcal R(A,B)$, then we can find $G_0\in\mathcal W\mathcal R(A_0,B)$ and an onto homomorphism $\widetilde\psi:G\rightarrow G_0$ which extends $\psi^{B}:A^{(B)}\rightarrow A_0^{(B)}$. Lemma \ref{special} implies the same for arbitrary, not necessarily onto, homomorphisms $\psi:A\rightarrow A_0$. Indeed, let $\rho:B\rightarrow A^B$ be as in Lemma \ref{special}. For $b,c\in B$, let $\tau_b=\psi^B(\rho_b)\in A_0^B$, $\beta_b=\text{Ad}(\tau_b)\sigma_b\in\text{Aut}(A_0^{(B)})$ and $w_{b,c}=\tau_b\sigma_b(\tau_b)\tau_{bc}^{-1}=\psi^B(v_{b,c})\in A_0^{(B)}$. 
Then $(\beta,w)$ is a cocycle action of $B$ on $A_0^{(B)}$, $G_0=A_0^{(B)}\rtimes_{\beta,w}B\in\mathcal W\mathcal R(A_0,B)$  and $\widetilde\psi(a,b)=(\psi^B(a),b)$, for $a\in A^{(B)}$ and $b\in B$, defines a homomorphism $\widetilde\psi:G\rightarrow G_0$ which extends $\psi^B:A^{(B)}\rightarrow A_0^{(B)}$.
\end{remark}

We end this section by using Lemma \ref{special} to characterise wreath-like products $G\in\mathcal W\mathcal R(A,B)$ as certain subgroups of the unrestricted wreath product group $A^B\rtimes_\sigma B$.

\begin{corollary}
Let $A,B$ be groups and denote by $\delta:A^B\rtimes_\sigma B\rightarrow B$ the quotient homomorphism. Then a group belongs to $\mathcal W\mathcal R(A,B)$ if and only if it is isomorphic to a subgroup $G$ of $A^B\rtimes_\sigma B$ which satisfies that  $A^{(B)}\subset G$, $\delta(G)=B$ and $\ker(\delta_{|G})=A^{(B)}$. 
\end{corollary}

{\it Proof.} If $x\in A^B,b\in B$, then $(x,b)(A^c)(x,b)^{-1}=A^{bc}$. This implies that any subgroup $G$ of $A^B\rtimes_\sigma B$ such that $A^{(B)}\subset G$, $\delta(G)=B$ and $\ker(\delta_{|G})=A^{(B)}$ belongs to $\mathcal W\mathcal R(A,B)$.

Conversely, let $G\in\mathcal W\mathcal R(A,B)$ and represent it as $G=A^{(B)}\rtimes_{\alpha,v}B$ as in Lemma \ref{special}. Then $\lambda:G\rightarrow A^B\rtimes_\sigma B$ given by $\lambda(a,b)=(a\rho_b,b)$, for every $a\in A^{(B)}, b\in B$, is an injective homomorphism. Since $A^{(B)}\subset\lambda(G)$, $\delta(\lambda(G))=B$ and $\ker(\delta_{|\lambda(G)})=A^{(B)}$, the conclusion follows.
\hfill$\blacksquare$

\section{Wreath-like product von Neumann algebras}\label{wrelike}
In this section, motivated by Corollary \ref{ex}, we introduce and study a notion of wreath-like product for tracial von Neumann algebras.

\begin{definition}
Let $(M,\tau)$ be a tracial von Neumann algebra and $B$ be a group. A tracial von Neumann algebra $N$ is said to be a {\it wreath-like product of $M$ and $B$} if it is isomorphic to $M^B\rtimes_{\beta,w}B$, where $(\beta,w)$ is a cocycle action of $B$ on $M^B$ such that $\beta_b(M^c)=M^{bc}$, for every $b,c\in B$. We denote by  $\mathcal W\mathcal R(M,B)$ the class of all wreath-like products of $M$ and $B$.
\end{definition}

\begin{example}
Let $A,B$ be groups and $G\in\mathcal W\mathcal R(A,B)$. Then $\text{L}(G)\in\mathcal W\mathcal R(\text{L}(A),B)$.
To see this, by Corollary \ref{wre} we have $G\cong A^{(B)}\rtimes_{\alpha,v}B$, for a cocycle action $(\alpha,v)$ on $B$ on $A^{(B)}$ such that $\alpha_b(A_c)=A_{bc}$, for every $b,c\in B$. 
Let $(u_g)_{g\in G}\subset\mathcal U(\text{L}(G))$ be the canonical generating unitaries.
For $b,c\in B$, let $\beta_b$ be the automorphism of $\text{L}(A^{(B)})=\text{L}(A)^B$ induced by $\alpha_b$ and let $w_{b,c}=u_{v_{b,c}}$.
Since $\text{L}(G)\cong\text{L}(A)^B\rtimes_{\beta,w}B$ and $\beta_b(M^c)=M^{bc}$, for every $b,c\in B$, it follows that $\text{L}(G)\in\mathcal W\mathcal R(\text{L}(A),B)$.
This example will be significantly generalized in the main result of this section, Proposition \ref{extend}.
\end{example}

\begin{lemma}\label{factor}
Let $(M,\tau)$ be a non-trivial tracial von Neumann algebra, $B$ an infinite group and $N\in\mathcal W\mathcal R(M,B)$.
Then $N$ is a II$_1$ factor. Moreover, if $M$ is a factor, then  $M^B\subset N$ is a regular irreducible subfactor. 
\end{lemma}
A subfactor $P\subset N$ is called {\it regular} if its normalizer, $\mathcal N_N(P)=\{u\in\mathcal U(N)\mid uPu^*=P\}$, generates $N$ and {\it irreducible} if $P'\cap N=\mathbb C1$.

{\it Proof.} We claim that $(M^B)'\cap N\subset M^B$. To this end, fix a sequence $(c_n)\subset B$ with $c_n\rightarrow\infty$.  Let $x\in (M^B)'\cap N$ and consider its Fourier decomposition $x=\sum_{b\in B}x_bu_b$, where $x_b\in M^B$, for $b\in B$. 

Let $b\in B\setminus\{e\}$. Then $x_b\beta_b(y)=yx_b$, for every $y\in M^B$ and $b\in B$. Let $y\in \mathcal U(M)\setminus\mathbb T1$. For $c\in B$, let $i_c:M\rightarrow M^c$ be the canonical isomorphism.
Let $y_n=i_{c_n}(y)\in\mathcal U(M^B)$.
As $y_n\in \mathcal U(M^{c_n}),\beta_b(y_n)\in \mathcal U(M^{bc_n})$, $\tau(y_n)=\tau(\beta_b(y_n))=\tau(y)$ and $bc_n\not=c_n$, for every $n$, we get that
$\|x_b\|_2^2=\langle x_b\beta_b(y_n),y_nx_b\rangle\rightarrow |\tau(y)|^2\|x_b\|_2^2.$ Since $|\tau(y)|<1$, this gives that $x_b=0$. Since this holds for every $b\in B\setminus\{e\}$, we derive that $x=x_e\in M^B$, which proves the claim. 

If $x\in N$ is central, then $x\in M^B$ by the claim.
Since $\beta_{c_n}(M^F)=M^{c_nF}$ and $\beta_{c_n}(x)=u_{c_n}xu_{c_n}^*=x$, for every $F\subset B$ and $n$, we derive that $\|x\|_2^2=\langle \beta_{c_n}(x),x\rangle\rightarrow |\tau(x)|^2$. Thus, $x\in\mathbb C1$, proving that $N$ is a factor. As $M^B\subset N$ is regular, the claim implies the moreover assertion, finishing the proof. 
\hfill$\blacksquare$

\begin{remark}\label{remark.relative.commutant}
Note that the proof of Lemma \ref{factor} shows that $(M^S)'\cap N\subset M^B$ for any infinite subset $S\subset B$.
\end{remark}

Our next goal is to give constructions of wreath-like product von Neumann algebras.

\begin{notation}\label{notatie}
Let  $(M,\tau)$ be a tracial von Neumann algebra and $B$ be a group. We denote by 
\begin{enumerate}
\item $\gamma:\mathcal U(M)^{(B)}\rightarrow\mathcal U(M^B)$ the homomorphism given by $\gamma((x_b)_{b\in B})=\otimes_{b\in B}x_b$.
\item $\eta:\mathcal U(M)^B\rightarrow\text{Aut}(M^B)$ the homomorphism given by  $\eta((y_b)_{b\in B})=\otimes_{b\in B}\text{Ad}(y_b)$.
\item $B\curvearrowright^{\sigma}\mathcal U(M)^B$ the shift action of $B$ (which preserves the subgroup $\mathcal U(M)^{(B)}<\mathcal U(M)^B$).
\item $B\curvearrowright^{\lambda} M^B$ the Bernoulli shift action given by $\lambda_b(x)=\otimes_{c\in B}x_{b^{-1}c}$, for $x=\otimes_{c\in B}x_c\in M^B$.
\end{enumerate}
\end{notation}

With this notation, we have:

\begin{lemma}\label{wrvN1}
 Let $(M,\tau)$ be a tracial von Neumann algebra and $B$ a group. Let $\xi:B\rightarrow\mathcal U(M)^B$ be a map such that $\xi_b\sigma_b(\xi_c)\xi_{bc}^{-1}\in\mathcal U(M)^{(B)}$, for every $b,c\in B$. Define  $\beta_b=\eta(\xi_b)\lambda_b\in\emph{Aut}(M^B)$ and $w_{b,c}=\gamma(\xi_b\sigma_b(\xi_c)\xi_{bc}^{-1})\in\mathcal U(M^B)$, for every $b,c\in B$. Then $(\beta,w)$ is a cocycle action of $B$ on $M^B$ and $M^B\rtimes_{\beta,w}B\in\mathcal W\mathcal R(M,B)$.
\end{lemma}

The proof of this lemma uses the following result, whose proof is straightforward.

\begin{lemma}\label{lambda}
$\lambda_b(\gamma(\zeta))=\gamma(\sigma_b(\zeta))$, $\eta(\sigma_b(\xi))=\lambda_b\eta(\xi)\lambda_b^{-1}$,  for all $b\in B$, $\zeta\in\mathcal U(M)^{(B)}$, $\xi\in\mathcal U(M)^B$.
\end{lemma}

{\it Proof of Lemma \ref{wrvN1}.} Since $\beta_b(M^c)=M^{bc}$, for every $b,c\in B$, it suffices to check that $(\beta,w)$ is a cocycle action. 
Let $b,c,d\in B$.
Using Lemma \ref{lambda} and that $\eta(x)=\text{Ad}(\gamma(x))$, for every $x\in\mathcal U(M)^{(B)}$, we get that $$\beta_b\beta_c=\eta(\xi_b)(\lambda_b\eta(\xi_c)\lambda_b^{-1})\lambda_{bc}=\eta(\xi_b)\eta(\sigma_b(\xi_c))\lambda_{bc}=\text{Ad}(w_{b,c})\eta(\xi_{bc})\lambda_{bc}=\text{Ad}(w_{b,c})\beta_{bc}.$$

Moreover, as $\eta(y)(\gamma(x))=\gamma(yxy^{-1})$, for every $x\in\mathcal U(M)^{(B)}$ and $y\in\mathcal U(M)^B$, Lemma \ref{lambda} gives  $$\beta_b(w_{c,d})=\eta(\xi_b)(\lambda_b(\gamma(\xi_c\sigma_c(\xi_d)\xi_{cd}^{-1})))=\eta(\xi_b)(\gamma(\sigma_b(\xi_c\sigma_c(\xi_d)\xi_{cd}^{-1})))=\gamma(\xi_b\sigma_b(\xi_c)\sigma_{bc}(\xi_b)\sigma_b(\xi_{cd})^{-1}\xi_b^{-1}).$$  This fact implies that $w_{b,c}w_{bc,d}=\beta_b(w_{c,d})w_{b,cd}$ and hence $(\beta,w)$ is a cocycle action.
\hfill$\blacksquare$

We now arrive at the main result of this section.
 Let $A,B$ be groups and $G\in\mathcal W\mathcal R(A,B)$.  
 Given a tracial von Neumann algebra $(M,\tau)$, we show that any homomorphism $\pi:A\rightarrow\mathcal U(M)$ extends to a homomorphism $\widetilde\pi:G\rightarrow\mathcal U(N)$,  for some $N\in\mathcal W\mathcal R(M,B)$. 

To this end, using Lemma \ref{special}, we write $G=A^{(B)}\rtimes_{\alpha,v}B$, where $(\alpha,v)$ is a cocycle action of $B$ on $A^{(B)}$ given by $\alpha_b=\text{Ad}(\rho_b)\sigma_b$ and $v_{b,c}=\rho_b\sigma_b(\rho_c)\rho_{bc}^{-1}$, for some map $\rho:B\rightarrow A^B$.

\begin{proposition}\label{extend}
Let  $\pi:A\rightarrow\mathcal U(M)$ be a homomorphism, where $(M,\tau)$ is a tracial von Neumann algebra. Define $\xi:=\pi^B(\rho_b)\in\mathcal U(M)^B$, for every $b\in B$. 
Then $\xi_b\sigma_b(\xi_c)\xi_{bc}^{-1}\in\mathcal U(M)^{(B)}$, for every $b,c\in M$.
Define  $\beta_b=\eta(\xi_b)\lambda_b\in\emph{Aut}(M^B)$ and $w_{b,c}=\gamma(\xi_b\sigma_b(\xi_c)\xi_{bc}^{-1})\in\mathcal U(M^B)$, for every $b,c\in B$. 

Then $(\beta,w)$ is a cocycle action of $B$ on $M^B$,  $N:=M^B\rtimes_{\beta,w}B\in\mathcal W\mathcal R(M,B)$ and there is a homomorphism $\widetilde \pi:G\rightarrow\mathcal U(N)$ given by $\widetilde\pi(x)=\gamma(\pi^B(x))=\otimes_{b\in B}\pi(x_b)$ and $\widetilde\pi(e,c)=u_c$, for every $x=(x_b)_{b\in B}\in A^{(B)}$ and $c\in B$.
In particular, if $\pi(A)''=M$, then $\widetilde\pi(A^{(B)})''=M^B$ and $\widetilde\pi(G)''=N$.
\end{proposition}

{\it Proof.} Since $\pi^B\circ\sigma_b=\sigma_b\circ\pi^B$, we get that $\xi_b\sigma_b(\xi_c)\xi_{bc}^{-1}\in\mathcal U(M)^{(B)}$, for every $b,c\in M$.
Lemma \ref{wrvN1} then implies that $(\beta,w)$ is a cocycle action of $B$ on $M^B$ and $N:=M^B\rtimes_{\beta,w}B\in\mathcal W\mathcal R(M,B)$.
Define $\widetilde\pi:A^{(B)}\rightarrow\mathcal U(M^B)$ by letting $\widetilde\pi(x)=\gamma(\pi^B(x))=\otimes_{b\in B}\pi(x_b)$, for every $x=(x_b)_{b\in B}\in A^{(B)}$.
Then $w_{b,c}=\gamma(\xi_b\sigma_b(\xi_c)\xi_{bc}^{-1})=(\gamma\circ\pi^B)(\rho_b\sigma_b(\rho_c)\rho_{bc}^{-1})=\widetilde\pi(v_{b,c})$, for every $b,c\in B$. Moreover, $(\widetilde\pi\circ\alpha_b)(x)=\otimes_{c\in C}\text{Ad}(\pi((\rho_b)_c))(\pi(x_{b^{-1}c}))=(\beta_b\circ\widetilde\pi)(x)$, for every $b\in B$ and $x=(x_c)_{c\in B}\in A^{(B)}$. Hence $\widetilde\pi\circ\alpha_b=\beta_b\circ\widetilde\pi$, for every $b\in B$. Altogether, Lemma \ref{extension} gives that $\widetilde\pi$ extends to a homomorphism $\widetilde\pi:G\rightarrow\mathcal U(N)$ such that $\widetilde\pi(e,b)=u_b$, for every $b\in B$. The last assertion is now immediate.
\hfill$\blacksquare$

\begin{notation}\label{not}
To emphasize the dependence on $\pi$, we hereafter write $(\beta^\pi,w^\pi)$ instead of $(\beta,w)$.
\end{notation}

\section{Rigidity for wreath-like product von Neumann algebras}
The following is the main technical result of this paper. 

\begin{theorem}\label{out}
Assume that $G\in\mathcal W\mathcal R(A,B)$ has property (T) and no nontrivial characters, where $A$ is a nontrivial countable group and $B$ is an icc hyperbolic group with $\emph{Out}(B)=\{e\}$.

For $i\in\{1,2\}$, let $\pi_i:A\rightarrow\mathcal U(M_i)$ be a homomorphism with $\pi_i(A)''=M_i$, where $M_i$ is a II$_1$ factor. 
Let $N_i=M_i^B\rtimes_{\beta^{\pi_i},w^{\pi_i}}B\in\mathcal W\mathcal R(M_i,B)$ be as defined in Proposition \ref{extend} and Notation \ref{not}.

Then $N_{i}$ is a II$_1$ factor with property (T) such that $\mathcal F(N_{i})=\{1\}$ and $\emph{Out}(N_{i})=\{e\}$.

Moreover, if $\theta:N_{1}\rightarrow pN_{2}p$ is a $*$-isomorphism, where  $p\in N_{2}$ is a projection, then $p=1$ and there is $u\in \mathcal U(N_{2})$ such that $\theta(\widetilde\pi_1(g))=u\widetilde\pi_2(g)u^*$, for every $g\in G$. Furthermore, there is a $*$-isomorphism $\theta_0:M_1\rightarrow M_2$ such that $\theta_0(\pi_1(g))=\pi_2(g)$, for all $g\in A$.

\end{theorem}

{\it Proof.}
The proof follows closely the strategy of the proof of \cite[Theorem 8.4]{CIOS21}. If $i\in\{1,2\}$, then Lemma \ref{factor} implies that  $N_{i}$ is a II$_1$ factor. Moreover, since $G$ has property (T) and the homomorphism $\widetilde\pi_i:G\rightarrow\mathcal U(N_{i})$ satisfies $\widetilde\pi_i(G)''=N_{i}$, we get that $N_{i}$ has property (T).
We only need to justify the moreover assertion since this implies that $\mathcal F(N_{i})=\{1\}$ and $\text{Out}(N_{i})=\{e\}$.

Let $\theta:N_{1}\rightarrow pN_{2}p$ be a $*$-isomorphism.
Since $M_2^B$ is a factor, we may assume that $p\in M_2^B$.
 We claim that
 \begin{equation}\label{inter}\text{$\theta(M_1^B)\prec_{pN_{2}p}^{\text{s}}pM_2^Bp$ and $pM_2^Bp\prec^{\text{s}}_{pN_{2}p}\theta(M_1^B)$.}\end{equation}

Since $N_{i}$ have property (T), $M_i$ is a factor and thus $M_i^B\subset N_{i}$ is an irreducible regular subfactor by Lemma \ref{factor}, for every $i\in\{1,2\}$, and $B$ is hyperbolic, \eqref{inter} follows from \cite[Lemma 8.5]{CIOS21}.  The proof of this lemma, which relies on \cite{PV12} (as used in \cite[Theorem 7.15]{CIOS21}), applies verbatim by using that for every $i\in\{1,2\}$  we have a $*$-homomorphism $\Delta_i:N_{i}\rightarrow N_{i}\overline{\otimes}\text{L}(B)$ given by $\Delta_i(xu_b)=xu_b\otimes u_b$, for $x\in M_i^B, b\in B$.

Since $M_1^B\subset N_{1}, pM_2^Bp\subset pN_{2}p$ are regular irreducible subfactors and the countable groups $\mathcal N_{N_{1}}(M_1^B)/\mathcal U(M_1^B)\cong\mathcal N_{pN_{2}p}(pM_2^Bp)/\mathcal U(pM_2^Bp)\cong B$ are icc,  \eqref{inter} together with \cite[Lemma 8.4]{IPP05} (see also \cite[Theorem 7.4]{CIOS21}) imply that there is $u\in\mathcal U(pN_2 p)$ such that $u\theta(M_1^B)u^*=pM_2^Bp$. After replacing $\theta$ by $\text{Ad}(u)\circ\theta$, we may thus assume that \begin{equation}\label{core}\theta(M_1^B)=pM_2^Bp.\end{equation}

Since $\mathcal N_{N_{1}}(M_1^B)/\mathcal U(M_1^B)\cong\mathcal N_{pN_{2}p}(pM_2^Bp)/\mathcal U(pM_2^Bp)\cong B$, there is an automorphism $\delta$ of $B$  such that $\theta(u_b)\in M_2^Bu_{\delta(b)}$, for every $b\in B$.
Since $\text{Out}(B)=\{e\}$, there is $c\in B$ such that $\delta(b)=cbc^{-1}$, for every $b\in B$. Thus, after replacing $\theta$ by $\text{Ad}(u_c^*)\circ\theta$, we get that \begin{equation}\label{discrete}\text{$\theta(u_b)\in M_2^Bu_b$, for every $b\in B$.}\end{equation}

Recall that  $G=A^{(B)}\rtimes_{\alpha,v}B$. By Proposition \ref{extend} we have that $\widetilde\pi_1(e,b)=\widetilde\pi_2(e,b)=u_b$, for every $b\in B$. 
If $a\in A^{(B)}$, then $\widetilde\pi_1(a,e)\in M_1^B$ and $\widetilde\pi_2(a,e)\in M_2^B$ and so $\theta(\widetilde\pi_1(a,e))\in M_2^B$ by \eqref{core}. By combining these facts with \eqref{discrete}, we deduce that \begin{equation}\label{cocycle} \text{$\theta(\widetilde\pi_1(g))\in M_2^B\widetilde\pi_2(g)$, for every $g\in G$.}\end{equation}

Consider the trace preserving action $G\curvearrowright^\kappa M_2^B$ given by $\kappa_g=\text{Ad}(\widetilde\pi_2(g))$ for $g\in G$. By \eqref{cocycle},  for any $g\in G$, there is $\eta_g\in M_2^B$ such that $\theta(\widetilde\pi_1(g))=\eta_g\widetilde\pi_2(g)$. Then $(\eta_g)_{g\in G}$ is a generalized $1$-cocycle for $\kappa$ with support projection $p$: $\eta_g\eta_g^*=p,\eta_g^*\eta_g=\lambda_g(p)$ and $\eta_{gh}=\eta_g\kappa_g(\eta_h)$, for every $g,h\in G$.

Next, Proposition \ref{extend} gives that $\widetilde\pi_2(x)=\otimes_{b\in B}\pi_2(x_b)$, thus $\kappa_{(x,e)}(M_2^c)=M_2^c$, for $x=(x_b)_{b\in B}\in A^{(B)}$ and $c\in B$.
If $b\in B$, then $\widetilde\pi_2(e,b)=u_b$, hence $\kappa_{(e,b)}=\text{Ad}(u_b)=\beta^{\pi_2}_b$ and thus $\kappa_{(e,b)}(M_2^c)=M_2^{bc}$, for every $b,c\in B$. Altogether, we get $\kappa_{(x,b)}(M_2^c)=M_2^{bc}$, for every $x\in A^{(B)}$ and $b,c\in B$. Equivalently, if $\varepsilon:G\rightarrow B$ is the quotient homomorphism, then $\kappa_g(M_2^c)=M_2^{\varepsilon(g)c}$, for every $g\in G$ and $c\in B$. Thus, $G\curvearrowright^\kappa M_2^B$
is built over the action $G\curvearrowright B$ given by $g\cdot b=\varepsilon(g)b$, for $g\in G, b\in B$, in the sense of \cite[Definition 2.5]{KV15}.

Since $G$ has no nontrivial characters, by applying \cite[Theorem 7.10]{CIOS21}, we get that $p=1$ and there is $u\in\mathcal U(M_2^B)$ such that $\eta_g=u\kappa_g(u)^*$. Thus, we derive that 
\begin{equation}\label{conjug}
\text{$\theta(\widetilde\pi_1(g))=u\widetilde\pi_2(g)u^*$, for every $g\in G$.}
\end{equation}

By Proposition \ref{extend} we get that $\widetilde\pi_i(g)=\pi_i(g)\in M_i^e$, for every $i\in\{1,2\}$ and $g\in (A)_e$.
Denote by $\theta_0$ the restriction of $\text{Ad}(u^*)\circ\theta$ to $\widetilde M_1^e$. Then \eqref{conjug} implies that $\theta_0(M_1^e)=M_2^e$ and identifying $M_i^e\equiv M_i$,  we have $\theta_0(\pi_1(g))=\pi_2(g)$, for every $g\in A$. This finishes the proof.
\hfill$\blacksquare$

\section{Embeddings into property (T) II$_1$ factors}
This section is devoted to the proofs of our main results.
To prove Theorems \ref{A} and \ref{B}, we will need the fact, observed in \cite[Remark 1.3]{GP98} and recorded below, that a separable McDuff II$_1$ factor can be generated by $3$ unitaries.
Although we will not use this, note that moreover any separable II$_1$ factor with property Gamma can be generated by $2$ unitaries by \cite[Theorem 6.2]{GP98}.

\begin{lemma}\emph{\cite{GP98}}\label{GP98}
Let $M$ be a separable II$_1$ factor and $R$ be the hyperfinite II$_1$ factor.
Let $(w_i)_{i\geq 1}\subset\mathcal U(M)$ be a sequence of unitaries which generate $M$.
 Let $u,v\in \mathcal U(R)$ be two unitaries which generate $R$  and $(p_i)_{i\geq 1}\subset R$ be nonzero projections such that $\sum_{i\geq 1}p_i=1$.

Then the unitaries $1\otimes u,1\otimes v,\sum_{i\geq 1}w_i\otimes p_i$ generate $M\overline{\otimes}R$. \end{lemma}

\subsection{Proof of Theorem \ref{A}} 
Let $(M,\tau)$ be a separable tracial von Neumann algebra.  Since $M$ embeds into a separable II$_1$ factor (e.g., we can take $M*\text{L}(\mathbb F_2)$),  part (2) of Theorem \ref{A} follows from Theorem \ref{C} proven below. Thus, it remains to prove part (1) of Theorem \ref{A}.

To this end, let $\widetilde M:=(M*\text{L}(\mathbb F_2))\overline{\otimes}R$. Since $M*\text{L}(\mathbb F_2)$ is a separable II$_1$ factor, by Lemma \ref{GP98} we find a homomorphism
$\pi:\mathbb F_3\rightarrow\mathcal U(\widetilde M)$  such that $\pi(\mathbb F_3)''=\widetilde M$.

Let $H$ be an acylindrically hyperbolic group with property (T). By Theorem \ref{wreath}(2), $H$ admits a quotient $K\in\mathcal W\mathcal R(\mathbb F_3,C)$, for some group $C$. Let $r:H\rightarrow K$ be the onto homomorphism.
 By Proposition \ref{extend}, there is a cocycle action $(\beta,w)$ of $C$ on $\widetilde M^C$ such that $Q:=\widetilde M^C\rtimes_{\beta,w}C$ belongs to $\mathcal W\mathcal R(\widetilde M,C)$ and we have a homomorphism $\widetilde\pi:K\rightarrow\mathcal U(Q)$ with $\widetilde\pi(K)''=Q$.
By Lemma \ref{factor}, $Q$ is a II$_1$ factor. Moreover, the homomorphism $\widetilde\pi\circ r:H\rightarrow\mathcal U(Q)$ satisfies $(\widetilde\pi\circ r)(H)''=\widetilde\pi(K)''=Q$. Therefore, if $H$ has property (T), so does $Q$. Since $M$ embeds into $\widetilde M$ and thus into $Q$, this finishes the proof of part (1) of Theorem \ref{A}.\hfill$\blacksquare$

\subsection{Proof of Corollary \ref{B}}
Let $(M,\tau)$ be a tracial von Neumann algebra.
If $M$ is separable, then it embeds into a II$_1$ factor with property (T) by Theorem \ref{A}. 
Conversely, assume that $M$ embeds into a property (T) II$_1$ factor $N$. Since $N$ has property (T), it is finitely generated by \cite[Theorem 4.4.1]{Po86} (see also \cite[Lemma 2.8]{HJKE21}) and thus separable. Hence, $M$ is separable as well.
\hfill$\blacksquare$

\subsection{Proof of Theorem \ref{C}}
By Theorem \ref{wreath}(1), there is a property (T) group $G\in\mathcal W(\mathbb F_3,B)$ with no nontrivial characters, for some icc hyperbolic group $B$ with $\text{Out}(B)=\{e\}$. 
 Let $R$ be the hyperfinite II$_1$ factor, $u,v\in \mathcal U(R)$ be two generating unitaries  and $(p_i)_{i\geq 1}\subset R$ be nonzero projections such that $\sum_{i\geq 1}p_i=1$.  

Let $M$ be a separable II$_1$ factor and $s\in \mathbb T$. Let $(w_i^{(M,s)})_{i\geq 1}\subset\mathcal U(M)$ be a sequence of unitaries which generate $M$ such that $\tau(w_1^{(M,s)})=s$. Let $a,b,c$ be free generators of $\mathbb F_3$.
Define a homomorphism  $\pi_{(M,s)}:\mathbb F_3\rightarrow\mathcal U(M\overline{\otimes}R)$ by letting $$\text{$\pi_{(M,s)}(a)=1\otimes u,\pi_{(M,s)}(b)=1\otimes v$ \;\;and \;\;$\pi_{(M,s)}(c)=\sum_{i\geq 1}w_i^{(M,s)}\otimes p_i$.}$$ By Lemma \ref{GP98} we have that $\pi_{(M,s)}(\mathbb F_3)''=M\overline{\otimes}R$.

Applying Proposition \ref{extend} and Theorem \ref{out} to $A=\mathbb F_3$ and $\pi_{(M,s)}$ gives a II$_1$ factor $P_{(M,s)}:=N_{\pi_{(M,s)}}$ which has property (T) and satisfies $\text{Out}(P_{(M,s)})=\{e\}$ and $\mathcal F(P_{(M,s)})=\{1\}$. Since $P_{(M,s)}$ belongs to  $\mathcal W\mathcal R(M\overline{\otimes}R,B)$, it contains $M\overline{\otimes}R$ and hence $M$.

Assume that $P_{(M,s)}$ and $P_{(N,t)}$ are stably isomorphic, for some separable II$_1$ factors $M,N$ and $s,t\in\mathbb T$. Theorem \ref{out} implies that there is a $*$-isomorphism $\theta_0:M\overline{\otimes}R\rightarrow N\overline{\otimes}R$ such that $\theta_0(\pi_{(M,s)}(g))=\pi_{(N,t)}(g)$, for every $g\in\mathbb F_3$.
Thus, $\theta_0(1\otimes u)=1\otimes u$ and $\theta_0(1\otimes v)=1\otimes v$. Hence $\theta_0(1\otimes x)=1\otimes x$, for every $x\in R$, and there is a $*$-isomorphism $\theta_1:M\rightarrow N$ such that $\theta_0=\theta_1\otimes\text{Id}_R$.
From this we deduce that $\sum_{i\geq 1}\theta_1(w_i^{(M,s)})\otimes p_i=\theta_0(\sum_{i\geq 1} w_i^{(M,s)}\otimes p_i)=\sum_{i\geq 1}w_i^{(N,t)}\otimes p_i$ and therefore $\theta_1(w_i^{(M,s)})=w_i^{(N,t)}$, for every $i\geq 1$. In particular $s=\tau(w_1^{(M,s)})=\tau(w_1^{(N,t)})=t$. This implies the conclusion of Theorem \ref{C}.
 \hfill$\blacksquare$

The final goal of this section is to prove Theorem \ref{D}. To this end, we will need the following analogue of Lemma \ref{GP98} for equivalence relations:

\begin{lemma} \label{genR}
Let $\mathcal R$ be an ergodic countable p.m.p. equivalence relation on a probability space $(X,\mu)$.  Let $\mathcal T$ be an ergodic hyperfinite p.m.p. equivalence relation on a non-atomic probability space $(Y,\nu)$. 

\begin{enumerate}
\item There exist $\varphi,\psi\in [\mathcal T]$ such that $\{u_\varphi,u_\psi\}''=\emph{L}(\mathcal T)$. 
\item There exist $\alpha,\beta,\gamma\in [\mathcal R\times\mathcal T]$ such that $\{u_\alpha,u_\beta,u_\gamma\}''=\emph{L}(\mathcal R\times\mathcal T)$.
\end{enumerate}
\end{lemma}

{\it Proof.} (1)
Let $(Y,\nu)=\prod_{n\geq 1}\big(\{0,1\},\frac{1}{2}(\delta_0+\delta_1)\big)$ and identify $\mathcal T$ with the orbit equivalence relation induced by the dyadic odometer $\varphi:Y\rightarrow Y$ given by addition by $\underline{1}=(1,0,\cdots,0,\cdots)$:
$$\varphi(1,\cdots,1,0,x_{k+1},x_{k+2},\cdots)=(0,\cdots, 0,1,x_{k+1},x_{k+2},\cdots).$$
Let $A_1=\{(x_k)_{k\geq 1}\in Y\mid x_1=1\}$ and $A_n=\{(x_k)_{k\geq 1}\in Y\mid x_1=\cdots=x_{n-1}=0,x_n=1\}$, for $n\geq 2$.
Then $\{A_n\}_{n\geq 1}$ is a measurable partition of $Y$ and $\varphi^{2^n}(A_n)=A_n$, for every $n\geq 1$. We define $\psi\in [\mathcal T]$ by letting $\psi_{|A_n}={\varphi^{2^n}}_{|A_n}$, for every $n\geq 1$. Denote $u=u_{\varphi},v=u_{\psi}\in\text{L}(\mathcal T)$, $p_0=1$ and $p_n=1_{A_n}\in\text{L}^{\infty}(Y)$, for every $n\geq 1$.
Also, denote $M=\{u,v\}''$. 

Since $\text{L}(\mathcal T)$ is generated by $u$ and $\{p_n\}_{n\geq 1}$, in order to prove that $M=\text{L}(\mathcal T)$, it suffices to argue that $p_n\in M$, for every $n\geq 0$. We will  prove this assertion by induction on $n$. 

Assume that $p_0,\cdots,p_{N-1}\in M$, for $N\geq 1$.
As $u\in M$ and $v=\sum_{n\geq 1}p_nu^{2^n}\in M$, we get that \begin{equation}\label{w}w:=\sum_{n\geq N}p_nu^{2^n}\in M.\end{equation} Denote $q_N:=\sum_{n\geq N+1}p_n=1_{\cup_{n\geq N+1}A_n}$.
We note that $\varphi^{2^N}$ leaves invariant and acts ergodically on the set $\cup_{n\geq N+1}A_n=\{(x_k)_{k\geq 1}\in Y\mid x_1=\cdots=x_N=0\}$.  
Thus, we have that \begin{equation}\label{ergodic}\text{$\lim\limits_{K\rightarrow\infty}\frac{1}{K}\sum_{k=1}^K\text{Ad}(u^{k2^N})(p_n)=\frac{\tau(p_n)}{\tau(q_N)}q_N=2^{N-n}q_N$ \;\; in $\|\cdot\|_2$, \;\; for every $n\geq N+1$.}
\end{equation}
Since $\frac{1}{K}\sum_{k=1}^K\text{Ad}(u^{k2^N})(w)\in M$, for every $K\geq 1$, and $\text{Ad}(u^{2^N})(p_N)=p_N$,  combining \eqref{w} and \eqref{ergodic} gives that $p_Nu^{2^N}+\sum_{n\geq N+1}2^{N-n}q_Nu^{2^n}\in M.$
Since $q_N=1-\sum_{n=1}^Np_n$ and $p_1,\cdots,p_{N-1}\in M$, we further derive that 
\begin{equation}\label{p_N}
p_N\big(u^{2^N}-\sum_{n\geq N+1}2^{N-n}u^{2^n}\big)\in M.
\end{equation}
Finally, note that $u^{2^N}-\sum_{n\geq N+1}2^{N-n}u^{2^n}\in M$ has right support equal to $1$. This is because there is a $*$-isomorphism $\{u\}''\cong\text{L}^{\infty}(\mathbb T)$ which sends $u$ to the identity function $z$ and the equation $z^{2^N}-\sum_{n\geq N+1}2^{N-n}z^{2^n}=0$, for $z\in\mathbb T$, forces $z^{2^N}=1$ and thus only has finitely many solutions.
In combination with \eqref{p_N}, this implies that $p_N\in M$, which finishes the proof of (1).

(2) By part (1), there are $\varphi,\psi\in [\mathcal T]$ with $\{u_\varphi,u_\psi\}''=\text{L}(\mathcal T)$. 
Let $\alpha=\text{Id}_X\times\varphi,\beta=\text{Id}_X\times\psi\in [\mathcal R\times\mathcal T]$.

Let $\mathcal R_0\subset\mathcal R$ be a hyperfinite ergodic subequivalence relation. By (1), there are $\zeta_1,\zeta_2\in[\mathcal R_0]\subset[\mathcal R]$ with $\{u_{\zeta_1},u_{\zeta_2}\}''=\text{L}(\mathcal R_0)$ and so $\text{L}^{\infty}(X)\subset\{u_{\zeta_1},u_{\zeta_2}\}''$. Let $\{\zeta_i\}_{i\geq 3}\subset [\mathcal R]$ be a sequence such that  $\{u_{\zeta_i}\}_{i\geq 3}\subset \{u_\zeta\}_{\zeta\in [\mathcal R]}$ is $\|\cdot\|_2$-dense. Then $\big(\text{L}^{\infty}(X)\cup\{u_\zeta\}_{\zeta\in [\mathcal R]}\big)''\subset\{u_{\zeta_i}\}_{i\geq 1}''$, thus $\{u_{\zeta_i}\}_{i\geq 1}''=\text{L}(\mathcal R)$.

Next, let $\{Y_i\}_{i\geq 1}$ be a measurable partition of $Y$ consisting of sets of positive measure. Define $\gamma\in [\mathcal R\times\mathcal T]$ by letting $\gamma_{|(X\times Y_i)}=\zeta_i\times\text{Id}_{Y_i}$, for every $i\geq 1$. Then $u_{\gamma}=\sum_{i\geq 1}u_{\zeta_i}\otimes\text{1}_{Y_i}$. 

Finally, let $N=\{u_\alpha,u_\beta,u_\gamma\}''$. 
Since $\{u_\alpha,u_\beta\}''=\mathbb C1\otimes\text{L}(\mathcal T)$, we get that $\mathbb C1\otimes\text{L}(\mathcal T)\subset N$. This further implies that $\{u_{\zeta_i}\otimes 1\}_{i\geq 1}\subset N$. Since $\{u_{\zeta_i}\}_{i\geq 1}''=\text{L}(\mathcal R)$, we also get that $\text{L}(\mathcal R)\otimes\mathbb C1\subset N$. Altogether, it follows that $N=\text{L}(\mathcal R\times\mathcal T)$, which proves (2).
\hfill$\blacksquare$

\subsection*{Proof of Theorem \ref{D}}
Let $\mathcal R$ be a countable p.m.p. equivalence relation on a probability space $(X,\mu)$. Let $\mathcal P$ be a countable ergodic p.m.p. equivalence relation on $(X,\mu)$ which contains $\mathcal R$. Let $\mathcal T$ be an ergodic hyperfinite p.m.p. equivalence relation on a probability space $(Z,\nu)$.
Since $\mathcal R\times\Delta_Z\subset\mathcal P\times\mathcal T$, it suffices to prove the conclusion of Theorem \ref{D} for $\mathcal P\times\mathcal T$ instead of $\mathcal R$. By Lemma \ref{genR}, after replacing
$\mathcal R$ by $\mathcal P\times\mathcal T$, we may assume that there is a homomorphism $\pi:A\rightarrow [\mathcal R]$ such that $\{u_{\pi(g)}\}_{g\in A}''=\text{L}(\mathcal R)$, where $A=\mathbb F_3$.

 In the rest of the proof, we will construct an equivalence relation $\mathcal S$ on $(X^B,\mu^B)$ which satisfies the conclusion.
First, by Theorem \ref{wreath}(1), there is a property (T) group $G\in\mathcal W(A,B)$ with no nontrivial characters, for some icc hyperbolic group $B$ with $\text{Out}(B)=\{e\}$.  Using Lemma \ref{special}, we write $G=A^{(B)}\rtimes_{\alpha,v}B$, where $(\alpha,v)$ is a cocycle action of $B$ on $A^{(B)}$ given by $\alpha_b=\text{Ad}(\rho_b)\sigma_b$ and $v_{b,c}=\rho_b\sigma_b(\rho_c)\rho_{bc}^{-1}$, for some map $\rho:B\rightarrow A^B$, where $\sigma:B\rightarrow\text{Aut}(A^B)$ denotes the shift action.

Define the equivalence relation $\mathcal R^{(B)}$ on $(X^B,\mu^B)$ by  $((x_b)_{b\in B},(y_b)_{b\in B})\in\mathcal R^{(B)}$ if $(x_b,y_b)\in\mathcal R$, for every $b\in B$, and the set $\{b\in B\mid x_b\not=y_b\}$ is finite.
Also, define $\kappa:[\mathcal R]^B\rightarrow\text{Aut}(\mathcal R^{(B)})$ by letting $\kappa(\varphi)(x)=(\varphi_b(x_b))_{b\in B}$, for every $\varphi=(\varphi_b)_{b\in B}\in [\mathcal R]^B$ and $x=(x_b)_{b\in B}\in X^B$.
Note that $\kappa([\mathcal R]^{(B)})\subset [\mathcal R^{(B)}]$. 
We denote by $B\curvearrowright^\lambda(X^B,\mu^B)$ the Bernoulli shift action and note that $(\lambda_b)_{b\in B}\subset\text{Aut}(\mathcal R^{(B)})$.
Define $\tau:B\rightarrow\text{Aut}(\mathcal R^{(B)})$ and $\omega:B\times B\rightarrow [\mathcal R^{(B)}]$ by letting
$$\text{$\tau_b=\kappa(\pi^B(\rho_b))\lambda_b$ and $\omega_{b,c}=\kappa(\pi^B(v_{b,c}))$, for every $b,c\in B$.}$$
Since $\lambda_b\kappa(\pi^B(x))\lambda_b^{-1}=\kappa(\pi^B(\sigma_b(x)))$, for every $b\in B$ and $x\in A^B$, it follows that $(\tau,\omega)$ is a cocycle action of $B$ on $\mathcal R^{(B)}$. 
If $b\in B\setminus\{e\}$ and we let $\varphi=\pi^B(\rho_b)\in [\mathcal R]^B$, then $\tau_b(x)=(\varphi_c(x_{b^{-1}c}))_{c\in B}$, for every $x=(x_c)_{c\in B}X^B$.
Thus, if $(\tau_b(x),x)\in\mathcal R^{(B)}$, then the set $\{c\in B\mid\varphi_c(x_{b^{-1}c})\not=x_c\}$ is finite.
Since $B$ is infinite, we get that $\mu^B(\{x\in X^B\mid (\tau_b(x),x)\in\mathcal R^{(B)}\})=0$, for every $b\in B\setminus\{e\}$.

Let $\mathcal S$ be the smallest equivalence relation on $(X^B,\mu^B)$ which contains $\mathcal R^{(B)}$ and the graph of $\tau_b$, for every $b\in B$.
Let $u:[\mathcal R^{(B)}]\rightarrow \mathcal U(\text{L}(\mathcal R^{(B)}))$ and $i:\text{Aut}(\mathcal R^{(B)})\rightarrow\text{Aut}(\text{L}(\mathcal R^{(B)}))$ be the homomorphisms defined in Subsection \ref{ER}.  By Lemma \ref{cocycleER}, $(\beta,w):=(i\circ\tau,u\circ\omega)$ is a cocycle action of $B$ on $\text{L}(\mathcal R^{(B)})$ and $\text{L}(S)\cong\text{L}(\mathcal R^{(B)})\rtimes_{\beta,w}B$.

Let $\xi_b=(u\circ\pi)^B(\rho_b)\in\mathcal U(\text{L}(\mathcal R))^B$, for $b\in B$. Then $\xi_b\sigma_b(\xi_c)\xi_{bc}^{-1}\in\mathcal U(\text{L}(\mathcal R))^{(B)}$,
  $\beta_b=\eta(\xi_b)\lambda_b$ and $w_{b,c}=\gamma(\xi_b\sigma_b(\xi_c)\xi_{bc}^{-1})$, for every $b,c\in B$, where $\eta,\gamma$ are defined as in Notation \ref{notatie} for $M=\text{L}(\mathcal R)$.
Since $(u\circ\pi)(A)''=\text{L}(\mathcal R)$, 
 Proposition \ref{extend} and Theorem \ref{out} imply that $\text{L}(\mathcal S)$ has property (T), $\text{Out}(\text{L}(\mathcal S))=\{e\}$ and $\mathcal F(\text{L}(\mathcal S))=\{1\}$.
Since $\mathcal S$ contains $\mathcal R^{(B)}$, it also contains $\mathcal R\times\Delta_{X^{B\setminus\{e\}}}$, where we identify $X^e$ with $X$. This finishes the proof.
\hfill$\blacksquare$

\section{Structural properties of infinitely generic II$_1$ factors}\label{structure}

In this section, we apply Theorem \ref{A} to obtain two new structural properties for the class of infinitely generic II$_1$ factors introduced in  \cite[Propositions 5.7, 5.10 and 5.14]{FGHS16} (see also  \cite[Fact 6.3.14]{AGKE20}). Throughout this section, as in \cite{Go20,AGKE20}, we assume the Continuum Hypothesis.

\begin{proposition}\emph{\cite{FGHS16}}\label{infgen}
There is a class of separable II$_1$ factors $\mathcal G$ satisfying the following:
\begin{enumerate}
    \item $\mathcal G$ is embedding universal: every separable II$_1$ factor embeds into an element of $\mathcal G$,
    \item  any embedding $\pi: Q_1 \hookrightarrow Q_2$, for some $Q_1,Q_2\in  \mathcal G$ is elementary, i.e., it extends to an isomorphism $Q_1^\omega\cong Q_2^\omega$, and
    \item $\mathcal G$ is the maximum class with properties (1) and (2).
\end{enumerate}
 The elements of $\mathcal G$ are called \emph{infinitely generic II$_1$ factors}. 
\end{proposition}

\subsection{Embeddings into ultraproducts}
A well-known question of Popa asks if every $R^\omega$-embeddable separable II$_1$ factor $M$ admits an embedding $\pi: M \hookrightarrow R^\omega$ such that the relative commutant $\pi(M)'\cap R^\omega$ is a factor. While this question has been answered positively in some instances (e.g., when $M = R$ \cite{DL69}, $M={\rm L}({\rm SL}_3(\mathbb Z))$ \cite{Po13}, and $M$ is any II$_1$ factor elementarily equivalent to $R$ \cite{AGKE20}), it remains wide open in general. In particular, it is open for property (T) II$_1$ factors.


A variation of Popa's question, where $R$ is replaced by  infinitely generic II$_1$ factors, was considered recently by Goldbring in \cite{Go20}.
 Specifically,  \cite[Theorem 2.18]{Go20} shows that if $Q\in \mathcal G$, then any  II$_1$ factor $M$ with property (T) admits an embedding in $Q^\omega$ with factorial relative commutant.  

By combining Theorem \ref{A} with properties of wreath-like product von Neumann algebras we extend \cite[Theorem 2.18]{Go20}  to all full separable II$_1$ factors $M$. More generally, we have:

\begin{theorem}\label{Go20} Let $Q$ be any infinitely generic II$_1$ factor. Then any separable II$_1$ factor $P_0$ such that $P_0'\cap P_0^\omega$ is a (possibly trivial) factor admits an embedding $\pi: P_0 \hookrightarrow Q^\omega$ such that  $\pi(P_0)'\cap Q^\omega$ is a factor.
\end{theorem}
{\it Proof.} 
By \cite[Lemma 2.12]{Go20}, given $Q_1,Q_2 \in \mathcal G$, a separable II$_1$ factor $M$ admits an embedding in $Q^\omega_1$ with factorial relative commutant if and only if $M$ admits an embedding in $Q^\omega_2$ with factorial relative commutant.
Thus, we only need to show that we can find $Q\in \mathcal G$ such that $P_0\subset Q$ and the diagonal embedding $P_0 \subset  Q^\omega$ satisfies that  $P_0'\cap Q^\omega$ is a factor.

Let $P =P_0\overline \otimes R$. By the proof of Theorem
\ref{A} there exists a II$_1$ factor $N= P^B \rtimes_{\alpha,v}B \in \mathcal W\mathcal R(P, B)$ with property (T),  where $B$ is an icc hyperbolic group. 

As $\mathcal G$ is embedding universal, one can find $Q \in \mathcal G$ with $N\subset Q$. Since $N$ has property (T), it has $w$-spectral gap in the sense of \cite{Po09} in any extension (in fact this characterizes property (T), see \cite{Ta22}) and thus \ $N'\cap Q^\omega= (N'\cap Q)^\omega$. As $Q$ is existentially closed by \cite[Proposition 5.16]{Go18} we also have $(N'\cap Q )'\cap Q=N$. Altogether, these facts imply  

\begin{equation}\label{sgap}(N'\cap Q^\omega)'\cap Q^\omega= ((N'\cap Q)^\omega)'\cap Q^\omega= ((N'\cap Q )'\cap Q)^\omega=N^\omega.\end{equation}

We identify $P^e$ with $P$. 
Since $P_0 \subset P\subset N\subset Q\subset Q^\omega$, passing to relative commutants gives that $P_0'\cap Q^\omega \supset (N'\cap Q^\omega) \vee (P^{B\setminus \{e\}}\overline \otimes R)^\omega$. Taking relative commutants again and using \eqref{sgap} we get

\begin{equation*} \begin{split}(P_0'\cap Q^\omega)'\cap Q^\omega \subset ((N'\cap Q^\omega) \vee (P^{B\setminus \{e\}}\overline \otimes R)^\omega)'\cap Q^\omega &=(N'\cap Q^\omega)'\cap Q^\omega \cap ((P^{B\setminus \{e\}} \overline \otimes R)^\omega)'\\& = (N \cap (P^{B\setminus \{e\}}\overline \otimes R)')^\omega.\end{split}\end{equation*}
Since  $N \cap (P^{B\setminus \{e\}}\overline \otimes R)'=P_0$ (see Remark \ref{remark.relative.commutant}), we deduce that  $(P_0'\cap Q^\omega)'\cap Q^\omega \subset P_0^\omega$ and thus $ \mathcal Z(P'_0\cap Q^\omega)\subset P_0'\cap P_0^\omega\subset P_0'\cap Q^\omega$. This entails that $\mathcal Z(P_0'\cap Q^\omega)\subset\mathcal Z(P_0'\cap P_0^\omega)$ and thus  $\mathcal Z(P'_0\cap Q^\omega)=\mathbb C1$, which finishes the proof. \hfill$\blacksquare$
\vskip 0.1in

\subsection{Super McDuff II$_1$ factors} 

A II$_1$ factor $M$ is called \emph{super McDuff} if its central sequence von Neumann algebra, $M'\cap M^\omega$, is a II$_1$ factor. Examples of super McDuff factors include the hyperfinite  II$_1$ factor, $R$, the tensor product  $N \overline{\otimes } R$, where $N$ is a full II$_1$ factor, and the infinite tensor product $\overline {\otimes}_{n\in \mathbb N} N_n$, where $N_n$ are full II$_1$ factors (see \cite[Section 6]{AGKE20}).   In the context of studying super McDuff factors from a model theoretic perspective, it was asked in \cite[Question 6.3.3]{AGKE20} whether there exists an existentially closed II$_1$ factor that is super McDuff. 

A separable II$_1$ factor $M$ is called {\it existentially closed} if for any separable tracial von Neumann algebra $N$ containing $M$, there is an embedding $\pi:N\hookrightarrow M^\omega$ whose restriction to $M$ is the diagonal embedding $M\hookrightarrow M^\omega$ \cite{FGHS16}. By \cite[Proposition 5.11]{FGHS16}, infinitely generic II$_1$ factors are existentially closed.

In this subsection, we provide a positive answer to \cite[Question 6.3.3]{AGKE20}  by showing that all infinitely generic II$_1$ factors are super McDuff (see Theorem \ref{igsMc}). To establish this, we first use Theorem \ref{A} to show the class of infinitely generic II$_1$ factors that are super McDuff is embedding universal. 

\begin{proposition}\label{sMc}
Any separable II$_1$ factor $P$ embeds into an infinitely generic II$_1$ factor $M$  which is super McDuff.
\end{proposition}  
\begin{proof} Pick $Q_1\in \mathcal G $ such that $P\subset Q_1$. By Theorem \ref{A} there exists a property (T) II$_1$ factor $N_1$ such that   $Q_1\subset N_1$. Since $\mathcal G$ is embedding universal there exists a II$_1$ factor $Q_2\in \mathcal G$ such that $N_1 \subset Q_2$. By induction, one can find an increasing  sequence $(N_n)_{n\in \mathbb N} $ of property (T) II$_1$ factors and an increasing sequence $(Q_n)_{n\in \mathbb N}$ of infinitely generic II$_1$ factors satisfying    
\begin{equation}\label{sequence}
    P\subset Q_1\subset N_1\subset Q_2\subset N_2\subset \cdots \subset Q_n\subset N_n\subset \cdots 
\end{equation}
Let $M$ be the inductive limit II$_1$ factor arising from the sequence \eqref{sequence}. Observe that by construction we have $M= \overline {\cup_n N_n}^{\text{wot}}= \overline{\cup_n Q_n}^{\text{wot}}$. Since for all $n$ we have $Q_n\in \mathcal G$, using \cite[Lemma 6.3.16]{AGKE20} it follows that $M\in \mathcal G$.  

In the remaining part we argue that $M$ is super McDuff. 
As $N_n$ has property (T), we have $N_n'\cap M^\omega =(N_n'\cap M)^\omega$ and by \cite[Proposition 5.16]{Go18} we also have $(N_n '\cap M)'\cap M=N_n$. Using these relations  and \cite[Lemma 2.4]{BCI15} we derive that \begin{equation*}\mathcal Z(N'_n\cap M^\omega)=\mathcal Z((N_n'\cap M)^\omega)=(\mathcal Z(N_n'\cap M))^\omega= ( ((N'_n\cap M)'\cap M)\cap N_n' )^\omega= (\mathcal Z (N_n))^\omega=\mathbb C 1.\end{equation*} Thus, $N_n'\cap M^\omega$ is a factor for all $n\in \mathbb N$. Using \cite[Proposition 6.3.12]{AGKE20} we conclude  that $M'\cap M^\omega$ is a factor. \end{proof}

Combining Proposition \ref{sMc} with model theoretic methods from  \cite{GH16} we obtain the following.   

\begin{theorem}\label{igsMc} Any infinitely generic II$_1$ factor is super McDuff.

\end{theorem}

\begin{proof} 
Let $Q$ be an infinitely generic II$_1$ factor. By Proposition \ref{sMc} there is a super McDuff,  infinitely generic II$_1$ factor $P$ together with an embedding  $\pi:Q\hookrightarrow P$. By Proposition \ref{infgen}(2), the embedding $\pi$ is elementary. Since $P$ is super McDuff and $Q$ is McDuff,  \cite[Proposition 4.12]{GH16} implies that $Q$ is super McDuff.
 \end{proof}






\end{document}